\documentclass[12pt]{article}

\usepackage{amssymb}
\usepackage{amsthm}
\usepackage{amsmath}

\newtheorem{theorem}{Theorem}
\newtheorem{lemma}[theorem]{Lemma}

\theoremstyle{definition}

\newtheorem*{remark*}{Remark}

\newcommand\EE{\mathbb{E}}
\newcommand\PP{\mathbb{P}}

\newcommand\g{\gamma}
\renewcommand\ge{\geqslant}
\renewcommand\geq{\geqslant}
\renewcommand\le{\leqslant}
\renewcommand\leq{\leqslant}
\newcommand\Bi{\mathrm{Bin}}
\newcommand\Var{\mathrm{Var}}
\newcommand\floor[1]{\lfloor #1\rfloor}
\newcommand\ceil[1]{\lceil #1\rceil}
\newcommand\eps{\varepsilon}

\begin{document}

\title{Dense Subgraphs in Random Graphs}
\author{Paul Balister \footnote{Department of Mathematical Sciences, University of Memphis, Memphis TN 38152, USA. Partially supported by NSF grant DMS 1600742}, B\'ela Bollob\'as 
\footnote{Department of Pure Mathematics and Mathematical Statistics, University of Cambridge,
Cambridge CB3 0WB, UK \emph{and} Department of Mathematical Sciences, 
University of Memphis, Memphis TN 38152, USA
\emph{and} 
London Institute for Mathematical Sciences, 35a South St., Mayfair, London W1K 2XF,
UK. Partially supported by NSF grant DMS 1600742 and MULTIPLEX grant no. 317532.
}
Julian Sahasrabudhe \footnote{Instituto Nacional de Matem\'{a}tica Pura e Aplicada (IMPA), 
Estr. Dona Castorina, 110 Jardim Botânico, Rio de Janeiro RJ 22460-320, Brasil \emph{and}
Peterhouse, University of Cambridge, Cambridge CB2 1RD, UK}, 
Alexander Veremyev\footnote{Industrial Engineering \& Management Systems, University of Central Florida, Orlando, Florida, USA.
Supported, in part, by the AFRL Mathematical Modeling and Optimization Institute}
}
\maketitle

\begin{abstract}
For a constant $\g\in[0,1]$ and a graph~$G$, let $\omega_\g(G)$ be the largest integer $k$
for which there exists a $k$-vertex subgraph of $G$ with at least $\g\binom{k}{2}$ edges.
We show that if $0<p<\g<1$ then $\omega_\g(G_{n,p})$ is concentrated on a set of two integers. 
More precisely, with $\alpha(\g,p)=\g\log\frac{\g}{p}+(1-\g)\log\frac{1-\g}{1-p}$, we show that 
$\omega_\g(G_{n,p})$ is one of the two integers closest to
$\frac{2}{\alpha(\g,p)}\big(\log n-\log\log n+\log\frac{e\alpha(\g,p)}{2}\big)+\frac{1}{2}$, with high probability. 
While this situation parallels that of cliques in random graphs, a new technique is required to handle the more complicated ways in which these ``quasi-cliques'' may overlap.
\end{abstract}

\section{Introduction}

Let $G=(V(G),E(G))$ be a simple, undirected graph where $V(G)$ denotes the set of \emph{vertices} of $G$ (sometimes called \emph{nodes}) and $E(G)$ denotes the set of \emph{edges}. A graph $G$ is said to be \emph{complete} if all possible edges are present: if $\{i,j\} \in E(G)$ for all $i, j\in V(G) ,\, i\ne j$. For a subset $S\subseteq V(G)$, we denote by $G[S]$ the subgraph of $G$ \emph{induced by} $S$: the graph with vertex set $S$ and edge set $\{\{i,j\} : i,j \in S \} \cap E(G)$. A \emph{clique} $C$ is a subset of $V(G)$ for which $G[C]$ is a complete graph~\cite{luce1949method}.

%A clique is one of the basic concepts of graph theory which has been studied in combinatorics, the cliques have been studied systematically already by the 40s - Turan's theorem, Ramsey's theorem, Erdos, szekres bounds on diagonal ramsey numbers

Cliques are a indispensable concept in the theory of graphs and have been extensively studied in various contexts, reaching back to the 1930s with the celebrated results of Ramsey \cite{Ramsey} and Tur\'{a}n \cite{turan}. In random graphs, cliques have also been a central topic of study 
with its roots in the pioneering work of Erd\H{o}s \cite{ErdosRamsey} on the probabilistic method. More recently, the concept of a clique has arisen naturally in sociometry~\cite{luce1949method} to model cohesive subgroups of tightly knit elements in a graph~\cite{Bomze}. For example, in social networks, where vertices correspond to ``actors'' and edges indicate relationships between actors~\cite{WassermanFaust94}, a clique represents a group of people any two of which have a certain kind of relationship (friendship, acquaintance, etc.) with each other~\cite{Mokken79}. Some of the earliest work on cliques, in the context of sociometry, is presented in ~\cite{luce1949method,Luce50,HararyRoss57}.

%Moreover, cliques have a number of ideal properties of cohesiveness and robustness. For example, each vertex is connected to all other vertices in a clique, the distance between any pair of vertices in a clique is one, a clique has maximum possible edge density, vertex or edge connectivity, etc. (I elect that we cut this out - not important enough for the intro)

%A clique is maximal if it is not a subset of a larger clique, and maximum if there is no larger clique in the graph. The cardinality (size) of the maximum clique in $G$ is denoted by $\omega(G)$ and in a graph theory is referred to as the \emph{clique number} of $G$, which has received significant attention from graph theorists and studied in various settings. 

%The \emph{maximum clique problem} seeks for a clique of maximum cardinality in $G$~\cite{Bomze,pardalos1994maximum}. From the computational point of view,  this problem is known to be $NP$-hard~\cite{Garey79}. 
%Due to its numerous applications, the maximum clique problem is one of the most important $NP$-hard problem which has been extensively studied in the literature. A detailed survey of results on the maximum clique problem and related references can be found in~\cite{Bomze,wu2015review}. %Since finding  \emph{clique number} is a computationally challenging task, theoretical bounds on the \emph{clique number} especially those that can be computed in polynomial time are very important in algorithmic design and pruning procedures. [[[ Let's put this in later ; I don't think we should have it in so early ]]

However, in real-world applications, the clique is not always the correct concept; often we do not care that \emph{all} edges are present in a particular subset, but only that the set is very ``well connected'', in some appropriate sense. Consequently, a number of relaxations of the notion of ``clique'' have appeared in the literature in recent years~\cite{Pattillo13b,komusiewicz2016multivariate}.

%using cliques for discovering  \emph{large} {cohesive clusters} might be impractical due to the fact that the definition of a clique is rather idealistic and, thus, can be too limiting. Consequently, a number of \emph{clique relaxation} concepts has appeared in the literature in recent years, which are based on the idea of ``relaxing'' certain clique defining property of interest  (see ~\cite{Pattillo13b,komusiewicz2016multivariate} and references therein).

One of the most popular and widely used clique relaxation models is the $\g$-\emph{quasi-clique}, where $\gamma \in [0,1]$ is a parameter~\cite{AbeResSud02a}. In particular, for $\g \in [0,1]$, we say that a subset $S \subseteq V(G)$ of a graph $G$ is a $\g$-quasi-clique if the graph $G[S]$, induced by $S$, has at least $\g \binom{|S|}{2}$ edges.

This concept was first defined by Abello, Pardalos \& Resende~\cite{AbelloA99a} who were interested in quasi-cliques in graphs representing telecommunications data. Later, the idea of ``dense clusters'' (a more general concept which includes $\g$-quasi-cliques) were studied in the context of molecular interaction networks described by Hartwell, Leland, Hopfield, John, Leibler, Stanislas, Murray and Andrew \cite{hartwell1999molecular} and further analyzed by Spirin and Mirny \cite{spirin2003protein}. They reported that dense subgraphs in molecular interaction networks correspond to meaningful modules or building blocks of molecular networks such as protein complexes or dynamic functional units.
The problem of finding large dense subgraphs have also appeared in a number of other domains including biology~\cite{Bader03,Bu03,Hu05,bastkowski2015minimum}, social network analysis~\cite{Crenson78,lee2016query,WassermanFaust94}, finance~
\cite{Boginski05,Huang09,saban2010analysis,Sim06} and data mining \cite{nguyen2017micro,tsourakakis2013denser}. 

Given the the myriad of instances for which the notion is useful, one would like to efficiently compute solutions to basic questions about 
quasi-cliques in a given graph: for example, ``what is the largest $\g$-quasi clique in (a given graph) $G$''? However, it comes as no surprise that the  computational problem of finding the largest quasi-clique in a given graph (along with many other such questions) is a hard computational problem, in general~\cite{Pattillo12} -- similar to the sister problem of finding large cliques in graphs \cite{Karp,haastad1999clique}. Moreover, the literature on \emph{exact} computational methods for this class of problems is extremely sparse and mostly focuses on the development and application of heuristic methods. It is therefore natural to study quasi-cliques in ``random'' or ``typical'' graphs, which may suffice for most applications, while allowing us to avoid the many hard computational barriers blocking the general problem.

%The currently available exact methods allow one to explicitly compute the maximum $\gamma$-quasi-clique size only in relatively small and sparse graphs .
%The performance of any inexact (approximate or heuristic) method where obtaining exact solutions is computationally challenging might be hard to evaluate.

To this end, we study the order of the largest $\g$-quasi-clique in the binomial random graph, a project initiated in a paper of Veremyev and Boginski  \cite{veremyev2012dense}. For a graph $G$, we let $\omega_\g(G)$ be the size of the largest subset of vertices of $G$ that induces a $\g$-quasi clique. Of course, $\omega_1(G) $ is the classical ``clique number'' of $G$, often denoted by $\omega(G)$. 

We prove that $\omega_\g(G_{n,p})$ is concentrated on two explicitly determined points, with high probability as $n \rightarrow \infty$, provided $0<p< \g <1$ are fixed real numbers. See Section~\ref{chap:MainResult} for a more careful statement of this result. 

Although these bounds are asymptotic, computational experiments suggest that they are quite accurate even for relatively small ($n=50, 100$) graphs generated using the $G_{n,p}$ model.
% in which $w_\gamma(G)$ could be computed in a reasonable time using currently available optimization solvers and known MIP (Mixed Integer Programming) models. 
For the results of these experiments, see Section~\ref{chap:CompExp}.

\section{Notation and statement of the main result} \label{chap:MainResult}

As usual write $[n]$ for the set
$\{1,\dots,n\}$ and $G_{n,p}$ for the binomial random graph on vertex set
$[n]$ with edge probability $p\in(0,1)$. We use the notation $O_n(1)$ to denote a quantity that is
bounded by a constant as $n$ tends to infinity and we use $o_n(1)$ to denote a quantity that tends
to zero as $n$ tends to infinity. We say that a sequence of events $E_n$ holds
\emph{with high probability} (henceforth whp) if $\PP(E_n)=1-o_n(1)$. For a graph $G$ we let $e(G)$ denote the
number of edges in the graph.

A complete subgraph on $k$ vertices will be called a $k$-\emph{clique}, and we define the
\emph{clique number} $\omega(G)$ of a graph $G$ to be the largest integer $k$ for which $G$ contains a
$k$-clique. The study of the clique number of $G_{n,p}$ was first carefully considered by
Matula~\cite{Matula1}, who noticed that the clique number of $G_{n,p}$ is concentrated on a small
set of values. These results were later strengthened by Grimmett and McDiarmid~\cite{GMcDCliques}
and then Bollob\'as and Erd\H{o}s~\cite{BECliques}, who showed that for fixed $0\le p\le 1$ the
clique number takes one of only \emph{two values}, whp (See also Theorem~11.1 in~\cite{BBRandomGs}). We prove that a similar phenomena persists for $\g$-quasi-cliques. However, a significant difficulty arises when controlling the concentration of the count of $\g$-quasi cliques directly. We tackle this issue by instead controlling a closely related random variable, which is more naturally handled.

We call $n$-vertex graph a $\g$-quasi-clique if
$e(G)\ge\g\binom{n}{2}$. For a graph $G$, we define $\omega_\g(G)$ to be the the largest integer
$k$ for which there exists a $\g$-quasi-clique subgraph of order~$k$. For $0<p<\g<1$, we show that $\omega_{\g}(G_{n,p})$ is concentrated on two points whp as $n \rightarrow \infty$.
\begin{theorem}\label{thm:MainTheorem}
 Let\/ $0<p<\g<1$ and\/ $\eps>0$ be fixed and define\/
 \[
  \alpha(\g,p):=\g\log\frac{\g}{p}+(1-\g)\log\frac{1-\g}{1-p}.
 \]
 Then
 \[
  \omega_\g(G_{n,p})-\frac{2}{\alpha(\g,p)}\Big(\log n-\log\log n+\log\frac{e\alpha(\g,p)}{2}\Big)
  \in (-\eps,1+\eps), 
 \] whp.
 In particular, $\omega_\g(G_{n,p})$ is one of the two integers closest to
 \[
  \frac{2}{\alpha(\g,p)}\Big(\log n-\log\log n+\log\frac{e\alpha(\g,p)}{2}\Big)+\frac{1}{2},
 \] whp.
\end{theorem}

As usual, the binary entropy function for $\g\in(0,1)$ is
\[
 h(\g):=\g\log\frac{1}{\g}+(1-\g)\log\frac{1}{1-\g}.
\]
We use the following consequence of Stirling's formula. If $\g\in(0,1)$ is fixed, we have
\begin{equation}\label{equ:BinomBounds}
 \binom{n}{\gamma n+O_n(1)}=e^{nh(\g)-\frac{1}{2}\log(n\g(1-\g))+O_n(1)}.
\end{equation}
%We also make use of the standard inequality
%\begin{equation}\label{equ:BinomBound2}
% \binom{n}{k}\le e^{nh(k/n)},
%\end{equation}
%which holds for all $0\le k\le n$.

We first set out to give an upper bound for $\omega_{\g}(G_{n,p})$ which holds with high probability. Let $X_k=X_{k,\g}(G_{n,p})$ be the random variable which counts the number of subgraphs of $G_{n,p}$ that are $\g$-quasi-cliques on $k$ vertices. We easily obtain an upper bound on $\omega_\g(G_{n,p})$ by bounding $\EE X_k$.
In preparation, we state a basic fact about binomial random variables.

\begin{lemma}\label{lem:BinomRVBound}
 Let\/ $0<p<\g<1$ be fixed and $N \rightarrow \infty$. We have
 \[
  \PP(\Bi(N,p)=\ceil{\g N})=e^{-N\alpha(\g,p)+O(\log N)},
 \]
 and
 \[
  \PP(\Bi(N,p)\ge\g N)=e^{-N\alpha(\g,p)+O(\log N)}.
 \]
\end{lemma}
\begin{proof}
We have
\begin{align*}
 \PP(\Bi(N,p)=\ceil{\g N})&=\binom{N}{\ceil{\g N}}p^{\ceil{\g N}}(1-p)^{\floor{(1-\g)N}}\\
 &=e^{Nh(\g)-\frac{1}{2}\log(N\g(1-\g))+N(\g\log p+(1-\g)\log(1-p))+O(\log\frac{p}{1-p})}\\
 &=e^{-N\alpha(\g,p)+O(\log N)}.
\end{align*}
The second result follows as, for $r\ge \g N>pN$, $\PP(\Bi(N,p)=r)$ is decreasing in $r$,
and hence
\[
 \PP(\Bi(N,p)=\ceil{\g N})\le \PP(\Bi(N,p)\ge\g N)\le (N+1)\PP(\Bi(N,p)=\ceil{\g N}).
\]
\end{proof}

%\begin{lemma}\label{lem:BinomRVBound}
% Let\/ $0<p<\g\le 1$ and let\/ $N$ be a positive integer. We have
% \[
%  \PP(\Bi(N,p)\ge\g N)\le\frac{\g(1-p)}{\g-p}\PP(\Bi(N,p)=\ceil{\g N}).
% \]
%\end{lemma}
%\begin{proof}
%Of course, $\PP(\Bi(N,p)=k)=\binom{N}{k}p^k(1-p)^{N-k}$. So if we set
%\[
% R(k)=\frac{\PP(\Bi(N,p)=k+1)}{\PP(\Bi(N,p)=k)}=\frac{p}{1-p}\cdot\frac{N-k}{k+1},
%\]
%we notice that $R(k)$ is decreasing as $k$ increases. Thus
%$R(k)\le\frac{p}{1-p}\frac{1-\g}{\g}=\mu<1$ for any $k\ge\g N$. As a result, we have
%\begin{align*}
% \PP(\Bi(N,p)\ge\gamma N)
% &=\sum_{L\ge 0}\PP(\Bi(N,p)=\ceil{\g N}+L)\\
% &\le\PP(\Bi(N,p)=\ceil{\g N})\sum_{L\ge 0}\mu^L\\
% &=\frac{\g(1-p)}{\g-p}\PP(\Bi(N,p)=\ceil{\g N}),
%\end{align*}
%as desired.
%\end{proof}

We now may establish an upper bound on $\omega_\g(G_{n,p})$, that holds whp, thus proving one of the inequalities implicit in the statement of Theorem~\ref{thm:MainTheorem}. In the following sections, we go on to show that the distribution of quasi-cliques (actually a subclass of these quasi-cliques) is sufficiently concentrated to prove Theorem~\ref{thm:MainTheorem}.

\begin{lemma}\label{lem:FirstMoment}
Let\ $0<p<\g<1$ and\/ $\eps>0$ be fixed. Then as $n \rightarrow \infty$
 \[\omega_\g(G_{n,p})<\frac{2}{\alpha(\g,p)}(\log n-\log\log n+\log\frac{e\cdot\alpha(\g,p)}{2})+1+\eps,\]
 whp.
\end{lemma}
\begin{proof}
With $X_k=X_{k,\g}(G_{n,p})$ and $S=\binom{k}{2}$ we have
\begin{align*}
 \EE X_k&=\binom{n}{k}\PP(\Bi(S,p)\ge\g S)\\
 &\le \frac{n^k}{k!}e^{-S\alpha(\g,p)+O(\log S)}\\
 &=e^{k(\log n-\frac{\alpha(\g,p)(k-1)}{2}-\log(k/e)+o_k(1))}
\end{align*}
Let $\kappa=\frac{2}{\alpha(\g,p)}(\log n-\log\log n+\log\frac{e\cdot\alpha(\g,p)}{2})+1+\eps$.
If $k=\ceil{\kappa}$ then
\[
 \log n-\frac{\alpha(\g,p)(k-1)}{2}-\log(k/e)+o_k(1)<-\frac{\eps\cdot\alpha(\g,p)}{2}+o_k(1)
\]
is negative for large enough $n$, and hence the expectation must tend to zero.
Thus we have $\PP(X_k>0)\le\EE X_k=o_n(1)$. The existence
of a $\g$-quasi-clique on $j>k$ vertices implies, by a simple averaging argument, that there exists
a $\g$-quasi-clique subgraph on $k$ vertices. Thus if $X_k=0$ then $X_j=0$ for all $j>k$. Hence
$\omega_\g(G_{n,p})<\kappa$ with high probability.
\end{proof}

\section{$\g$-flat subgraphs}

To show that $G_{n,p}$ contains a $\g$-quasi-clique of order roughly
$\frac{2}{\alpha(\g,p)}\log n$ whp, we count a slightly restricted class of subgraphs. The advantage of
working with this restricted class is that the second moment of their count is controlled more naturally. 
Roughly speaking, we say that a $\g$-quasi-clique $G$ is $\g$-flat if every induced subgraph of $G$ is close to being a $\g$-quasi clique.

To make this definition precise, we need a few definitions. First, for a graph $G$ and a subset $A$ of the vertex set of~$G$, let us define $e(A)$ to be the number of edges with both end-points in~$A$.

Now, for $\g\in(0,1)$ and $\ell\in[k]$, we define $S=\binom{k}{2}$, $T=\binom{\ell}{2}$, and set
\[
 D_k(\ell)=\min(T,S-T)\ell^{-1/2}\log k.
\]
Call an $k$-vertex graph $G$ $\g$-\emph{flat} if $e(G)=\ceil{\g\binom{k}{2}}$ and for all
$A\subseteq V(G)$ with $\ell=|A|\in[2,k-1]$, we have $e(A)\le\g\binom{\ell}{2}+D_k(\ell)$.
We note that $\min(T,S-T)$ is clearly an upper bound on $e(A)-\g\binom{\ell}{2}$ when
$e(G)=\ceil{\g\binom{k}{2}}$, so this is only a restriction on $e(A)$ when $|A|=\ell>(\log k)^2$.

We shall show that if a subset of $k$ vertices in $G_{n,p}$ has $\ceil{\g\binom{k}{2}}$ edges then
it is reasonably likely that it will also be $\g$-flat, and hence the two notions are ``typically''
interchangeable. For positive integers $n$, $m$, $0\le m\le\binom{n}{2}$, we define the
\emph{Erd\H{o}s-R\'enyi random graph} $G(n,m)$ as the uniform probability space that is supported
on all $n$ vertex graphs with exactly $m$ edges.

\begin{lemma}\label{lem:typicallyflat}
 Let\/ $G=G(k,\ceil{\g\binom{k}{2}})$ and let $\g$ be fixed and $k \rightarrow \infty$. Then $G$ is\/ $\g$-flat with high probability.
\end{lemma}
\begin{proof}
Let $G=G(k,\ceil{\g\binom{k}{2}})$ be realized on the vertex set\/~$[k]$ and
fix a subset\/ $A\subseteq [k]$ with\/ $\ell=|A|\in[2,k-1]$. We shall show that
\begin{equation}\label{eq:flatbound}
 \binom{k}{\ell}\PP\left(e(A)\ge\g\tbinom{\ell}{2}+D_k(\ell)\right)\le k^{-2}.
\end{equation}
Set $S=\binom{k}{2}$, $T=\binom{\ell}{2}$, $R=S-T$, and put $C(L)=\PP(e(A)=L)$. Note that
\[
 C(L)=\binom{T}{L}\binom{R}{\ceil{\g S}-L}\binom{S}{\ceil{\g S}}^{-1}
\]
and for $0\le L<T$, we have
\begin{equation}\label{equ:lineIncrease}
 Q(L):=\frac{C(L+1)}{C(L)}
 =\frac{T-L}{L+1}\left(\frac{\ceil{\g S}-L}{R-\ceil{\g S}+L+1}\right).
\end{equation}
From \eqref{equ:lineIncrease} we see that $Q(L)$ is strictly decreasing as $L$ increases.
Let $L=\ceil{\g T}+r\le T$. Then if $r \geq 0$, 
\begin{align*}
 Q(L)&\le Q(\g T+r)\le\left(\frac{(1-\g)T-r}{\g T+r+1}\right)\left(\frac{\g R+1-r}{(1-\g)R+r}\right)\\
 &=\left(\frac{1-\frac{r}{(1-\g)T}}{1+\frac{r+1}{\g T}}\right)\left(\frac{1-\frac{r-1}{\g R}}{1+\frac{r}{(1-\g)R}}\right)\\
 &\le \min\Big\{1-\frac{r}{(1-\g)T},1-\frac{r-1}{\g R}\Big\}\le e^{-\frac{c(r-1)}{\min(R,T)}},
\end{align*}
where $c=1/\max(\g,1-\g)>0$ is a constant. Hence 
\begin{equation}\label{equ:A(D(l))}
 C(L) = C(\ceil{ \g T } + r) \ \le C(\ceil{\g T}+1)\prod_{s=1}^r e^{-\frac{c(s-1)}{\min(R,T)}}\le e^{-\frac{cr(r-1)}{2\min(R,T)}},
\end{equation}
where we have used the (trivial) fact that $C(\ceil{\g T}+1)\le 1$.
Now $T\ell^{-1/2}\log k\ge \frac{1}{2}\log k$ and
$R\ell^{-1/2}\log k\ge (k-1)^{1/2}\log k$, so $D_\ell(k)\to\infty$ uniformly in $\ell$
as $k\to\infty$. Thus for large $k$ we have
$cr(r-1)/(2\min(R,T))\ge c'\min(R,T)\ell^{-1}(\log k)^2$
for some $c'>0$ when $r>D_k(\ell)-1$. Hence
\begin{align}
 \PP\big(e(A)\ge\g T+D_k(\ell)\big)&=\sum_{\g T+D_k(\ell)\le L\le T} C(L)\notag\\
 &\le \ell^2 e^{-c'\min(R,T)\ell^{-1}(\log k)^2}\label{equ:MainProbinDevLem}
\end{align}
for large enough $k$.

Consider the case when $R<T$. Then $R=(k-\ell)(k+\ell-1)/2>\ell(k-\ell)/2$
and so $c'\min(R,T)\ell^{-1}(\log k)^2>5(k-\ell)\log k\ge (k-\ell)\log k+4\log k$
for large enough $k$. Now $\binom{k}{\ell}=\binom{k}{k-\ell}\le k^{k-\ell}$, so
\[
 \binom{k}{\ell}\PP\left(e(A)\ge\g\tbinom{\ell}{2}+D_k(\ell)\right)\le
 \binom{k}{\ell}k^2 e^{-(k-\ell)\log k-4\log k}\le k^{-2} ,
\] as required. Now suppose $R\ge T$. Then
$c'\min(R,T)\ell^{-1}(\log k)^2>3\ell\log k\ge\ell\log k+4\log k$
when $k$ is large enough. Now $\binom{k}{\ell}\le k^\ell$, so
\[
 \binom{k}{\ell}\PP\left(e(A)\ge\g\tbinom{\ell}{2}+D_k(\ell)\right)\le
 \binom{k}{\ell}k^2 e^{-\ell\log k-4\log k}\le k^{-2},
\]
as required. Hence \eqref{eq:flatbound} holds for all $\ell\in[2,k-1]$.

Now, for $2\le\ell\le k-1$, let $Y_\ell$ be the random variable counting the number of subsets
$A$ of order $\ell$ which induce more than $\g\binom{\ell}{2}+D_k(\ell)$ edges.
By \eqref{eq:flatbound} we have
\[
 \PP(Y_\ell>0)\le\EE(Y_\ell)\leq \binom{k}{\ell}\PP\left(e(A)\ge\g\tbinom{\ell}{2}+D_k(\ell)\right)
 \le k^{-2},
\]
for large enough~$k$. So the probability that $Y_\ell>0$ for any of the $<k$ choices for $\ell$
is at most $k^{-1}=o_k(1)$.
\end{proof}

Let $Z_k = Z_{k,n}$ be the random variable counting the number of copies of $\g$-flat subgraphs of order $k$ in $G_{n,p}$, with $p$ fixed and $n\rightarrow \infty$.
We now easily bound $\EE Z_k$, by using Lemma~\ref{lem:typicallyflat}, to relate it to the
quantity $\EE X_k$.

\begin{lemma}\label{lem:firstmomentZ}
 Let\/ $\eps>0$ and\/ $k\le \frac{2}{\alpha(\g,p)}(\log n-\log\log n+\log\frac{e\cdot\alpha(\g,p)}{2})+1-\eps$
 with\/ $k\to\infty$ as\/ $n\to\infty$. Then\/ $\EE Z_k\to\infty$.
\end{lemma}
\begin{proof}
We apply Lemma~\ref{lem:typicallyflat} to deduce that
\[
 \EE Z_k=\binom{n}{k}\PP(G(k,p)\text{ is $\g$-flat})\ge (1+o_k(1))\binom{n}{k}\PP(e(G(k,p))=\ceil{\g S}).
\]
Now $k=O(\log n)$ by assumption, so $\binom{n}{k}=\frac{n^k}{k!}(1-O(k^2/n))=(1+o_k(1))\frac{n^k}{k!}$. Hence
\begin{align*}
 \EE Z_k&=(1+o_k(1))\frac{n^k}{k!}\PP(\Bi(S,p)=\ceil{\g S})\\
 &=\frac{n^k}{k!}e^{-S\alpha(\g,p)+O(\log S)}\\
 &=e^{k(\log n-(k-1)\alpha(\g,p)/2-\log(k/e)+o_k(1))}.
\end{align*}
However, the exponent in the last line tends to infinity when $k\to\infty$ and
$k\le\frac{2}{\alpha(\g,p)}(\log n-\log\log n+\log\frac{e\alpha(\g,p)}{2})+1-\eps$.
\end{proof}

In the next section we turn to estimate the variance of $Z_k$.

\section{The second moment}

To prove our lower bound on $\omega_\g(G_{n,p})$, we count the number of $\g$-flat subsets of order~$k$ in $G_{n,p}$,
where $k$ is roughly $\frac{2}{\alpha(\g,p)}\log n$. For $k\in [n]$, recall that $Z_k$ is the random variable which
counts the number of $\g$-flat subsets of $G(n,p)$. To apply Chebyshev's inequality, we aim to
estimate the fraction
\begin{equation}\label{equ:FracF}
 F=\frac{\Var Z_k}{(\EE Z_k )^2}=\frac{\EE Z^2_k-(\EE Z_k)^2}{(\EE Z_k)^2}.
\end{equation}
In particular, we shall show $F=o(1)$, as both $k$ and $n$ tend to infinity. Let $A,B\subseteq [n]$ with $|A|=|B|=k$ and $|A\cap B|=\ell$.
We think of $\ell\in[2,k-1]$ and treat the degenerate cases $\ell \in\{0,1,k\}$ separately.
Put $S=\binom{k}{2}$, $T=\binom{\ell}{2}$, $R=S-T$ and let $g_\ell(L)$ denote the probability that
$e(A)=\ceil{\g S}$, $e(B)=\ceil{\g S}$ and $e(A\cap B)=L$.
We note that
\[
 g_\ell(L)=\binom{T}{L}\binom{R}{\ceil{\g S}-L}^2p^{2\ceil{\g S}-L}(1-p)^{2\floor{(1-\g)S}-T+L}
\]
and consider the ratio
\begin{align}
 R_\ell(L)&=\frac{g_\ell(L)}{\PP(e(A)=\ceil{\g S})^2}\notag\\
 &=\binom{T}{L}\binom{R}{\ceil{\g S}-L}^2\binom{S}{\ceil{\g S}}^{-2}p^{-L}(1-p)^{L-T}.\label{equ:R(L)}
\end{align}

The following lemma gives us a suitable way of estimating the quantity $R_\ell(L)$, for our purposes. For the remainder of the section, we maintain the assumption that $0 < p < \g \leq 1$ and that $k \rightarrow \infty$.

\begin{lemma}\label{lem:BoundOnRl}
 Let $2\le\ell\le k-1$, $r\geq 0$ be an integer and 
 set\/ $\lambda=2\cdot\frac{\g}{1-\g}\frac{1-p}{p}$.
 Then
 \begin{equation}\label{equ:BoundonRl}
  R_\ell(\floor{\g T}+r)\le\lambda^r e^{T\alpha(\g,p)+O_k(1)}
 \end{equation} and
 \begin{equation}\label{equ:BoundonRlneg}
  R_\ell(\floor{\g T}-r)\le R_\ell(\floor{\g T}).
 \end{equation}
\end{lemma}
\begin{proof}
We first bound $R_\ell(\floor{\g T})$. Note that $R\ge k-1$ and hence $R^2T=R^2(S-R)\ge S^2$
for $2\le \ell\le k-1$. Since $S=\binom{k}{2}\to\infty$ and $\g$ is fixed,
we may bound line~\eqref{equ:R(L)} by using equation~\eqref{equ:BinomBounds}, to obtain
\begin{align*}
 R_\ell(\floor{\g T})
 &=e^{Th(\g)+2(S-T)h(\g)-2Sh(\g)+\frac{1}{2}\log\frac{S^2}{\g(1-\g)R^2T}+O_k(1)}
  p^{-\floor{\g T}}(1-p)^{-\ceil{(1-\g)T}}\\
 &\le e^{-Th(\g)+O_k(1)}p^{-\floor{\g T}}(1-p)^{-\ceil{(1-\g)T}}\\
 &=e^{T\alpha(\g,p)+O_k(1)}.
\end{align*}
Now put $C(L)=R_\ell(L+1)/R_\ell(L)$ and observe that $C(L)$ can be written as
\[
 \frac{1-p}{p}\cdot\frac{T-L}{L+1}\left(\frac{\ceil{\g S}-L}{R-\ceil{\g S}+L+1}\right)^2.
\]
From this expression, we see that $C(L)$ strictly decreases as $L$ increases and therefore
\begin{align*}
 C(\floor{\g T}+r)&\le C(\floor{\g T})\\
 &\le\frac{1-p}{p}\cdot\frac{\g}{1-\g}\left(1+\frac{1}{\g R}\right).
\end{align*}
Now note that since $\ell<k$, by assumption, we have that $R\ge k-1$ and thus $R$ tends to
infinity with~$k$. Hence, for large $k$,
\[
 C(\floor{\g T}+r)\le 2\frac{1-p}{p}\frac{\g}{1-\g}=\lambda.
\]
We now apply this inequality $r$ times to obtain
\[
 R_\ell(\floor{\g T}+r)\le\lambda^r R_\ell(\floor{\g T}),
\]
which holds for $k$ sufficiently large, but independently of~$r$. This proves the
inequality~\eqref{equ:BoundonRl}.

To prove the inequality \eqref{equ:BoundonRlneg} we note that $C(L)$ is strictly decreasing and
\begin{align*}
 C(\floor{\g T}-1)&\ge\frac{1-p}{p}\frac{\g}{1-\g}\cdot\left(1+\frac{1}{(1-\g)T}\right)\left(1+\frac{1}{\g T}\right)\\
 &\ge\frac{1-p}{p}\frac{\g}{1-\g}>1.
\end{align*}
Thus $R_\ell(\floor{\g T}-r)\le R_\ell(\floor{\g T})$.
\end{proof}

We are now in a position to show that $F=o(1)$ as $n$ and $k$ tend to infinity.

\begin{lemma}\label{lem:VarianceCalc}
 Let\/ $k\le \frac{2}{\alpha(\g,p)}(\log n-\log\log n+\log\frac{e\alpha(\g,p)}{2})+1-\eps$.
 We have
 \[
  \EE Z_k^2=(1+o_k(1))(\EE Z_k)^2.
 \]
\end{lemma}
\begin{proof}
We consider the fraction $F$, from equation \eqref{equ:FracF}. We keep with the convention that
$S=\binom{k}{2}$, $T=\binom{\ell}{2}$, and $R=S-T$. Let $E_A$ and $E_B$ denote the events
that $A$, resp.\ $B$, induces a $\g$-flat subgraph. Let $E'_A$, resp. $E'_B$, denote the event
that $A$, resp.\ $B$, induce exactly $\ceil{\g S}$ edges. Note that $E_A\subseteq E'_A$ and
$E_B\subseteq E'_B$. Now write $t(\ell)=t_{n,k}(\ell)=\binom{k}{\ell}\binom{n-k}{k-\ell}\binom{n}{k}^{-1}$.
%and bound
%\begin{equation} \label{equ:boundOnt}
% t(\ell)\le(1+o_n(1))(k^2/n)^\ell.
%\end{equation}
We now turn to bound $F$.
We may expand $Z_k$ as a sum of indicators
\[ Z_k =  \sum_{A \subset V(G), |A| = k} \mathbf{1}(E_A) \]
Hence $\EE Z_k = \binom{n}{k}\PP(E_A)$ thus
\[  F = \frac{\EE Z_k^2 - (\EE Z_k)^2}{\EE Z_k} = \sum_{A,B} \binom{n}{k}^{-2} \frac{\PP(E_A \cap E_B ) - \PP(E_A)\PP(E_B)}{\PP(E_A)^2}.  \]
We now divide the sum with respect to $|A\cap B| = \ell$ to obtain
\begin{align}
 F&=\sum_{\ell=0}^k\binom{n}{k}\binom{k}{\ell}\binom{n-k}{k-\ell}\binom{n}{k}^{-2}
 \frac{\PP(E_A\cap E_B)-\PP(E_A)\PP(E_B)}{\PP(E_A)\PP(E_B)}\notag\\
 &=\sum_{\ell=2}^{k-1}t(\ell)\cdot\frac{\PP(E_A\cap E_B)-\PP(E_A)^2}{\PP(E_A)^2}+o_n(1),\label{equ:ExpandF}
\end{align}
where we have eliminated the first two terms in the above sum as $E_A$ and $E_B$ are independent
events when $|A\cap B|\le 1$. We have also  eliminated the last term in the sum, i.e. when $E_A=E_B$. This is justified, as this term is at most
$(\binom{n}{k}\PP(E_A)))^{-1}=(\EE Z_k)^{-1}=o_k(1)$, by Lemma~\ref{lem:firstmomentZ}.
Let us denote the $\ell$th term in the sum at \eqref{equ:ExpandF} as $F(\ell)$.

Lemma~\ref{lem:typicallyflat} implies that
\[
 \frac{\PP(E_A\cap E_B)-\PP(E_A)^2}{P(E_A)^2}\le(1+o_k(1))\frac{\PP(E_A\cap E_B)}{\PP(E'_A)^2}.
\]
%We now consider two cases; when $\ell\in [k^{2/3},k-1]$ and when $\ell\in [2,k^{2/3}]$.
For $\ell\in [2, k-1]$, our ``flatness condition'' on subsets of $A$ applies and hence
\begin{align*}
 \PP(E_A\cap E_B)/\PP(E'_A)^2&=\PP(E'_A)^{-2}\sum_{0\le L\le\g T+D_k(\ell)}\PP(E_A\cap E_B\mid e(A\cap B)=L)\PP(e(A\cap B)=L)\\
 &\le\PP(E'_A)^{-2}\sum_{0\le L\le\g T+D_k(\ell)}\PP(E'_A\cap E'_B\mid e(A\cap B)=L)\PP(e(A\cap B)=L)\\
 &=\sum_{0\le L\le\g T+D_k(\ell)}R_\ell(L)\\
 &=\sum_{0\le L <\g T}R_\ell(L)+\sum_{\g T\le L\le\g T+D_k(\ell)} R_\ell(L)\\
 &\le T\lambda^{D_k(\ell)}e^{T\alpha(\g,p)+O_k(1)}.
\end{align*}
This last inequality follows from applying the inequality~\eqref{equ:BoundonRl}
(from Lemma~\ref{lem:BoundOnRl}) to each term in the right sum and applying the
inequality~\eqref{equ:BoundonRlneg} (again from Lemma~\ref{lem:BoundOnRl}) to the left sum.
So we may bound the $\ell$th term in the sum \eqref{equ:ExpandF}
as
\[
 F(\ell)\le t(\ell)T\lambda^{D_k(\ell)}e^{T\alpha(\g,p)+O_k(1)}.
\]
We first consider the case when $R<T$. Write $\delta:=k-\ell$.
Now
\[
 t(\ell)=\binom{k}{\delta}\binom{n-k}{\delta}\binom{n}{k}^{-1}\le (kn)^\delta\binom{n}{k}^{-1}
\]
and $\EE Z_k=\binom{n}{k}e^{-S\alpha(\g,p)+O(\log k)}\to\infty$.
Also $D_k(\ell)=R\ell^{-1/2}\log k=o_k(R)$ as $R<T$ implies $\ell\ge k/2$.
Thus
\[
 F(\ell)\EE Z_k\le e^{\delta\log(kn)-R\alpha(\g,p)+o_k(R)}
\]
But $R=\delta(k+\ell-1)/2\ge 2k\delta/3$ and $k\alpha(\g,p)\sim 2\log n$. Thus
$F(\ell)\le (\EE Z_k)^{-1}e^{-(\frac{1}{3}-o_k(1))\delta\log n}$.
In particular, $\sum_{\ell\colon R<T}F(\ell)=o(1)$.

Now consider the case when $R\ge T$. In this case we use the bound
$t(\ell)\le (1+o_k(1))(k^2/n)^\ell$ to deduce that
\[
 F(\ell)\le e^{T\alpha(\g,p)+\ell\log(k^2/n)+O_k(T \ell^{-1/2}\log k)+O_k(\log k)}
 =e^{\ell((\ell-1)\alpha(\g,p)/2-\log(n)+O_k(k^{1/2}\log k))}.
\]
Now $k=O(\log n)$ and $\ell<3k/4$. Thus $(\ell-1)\alpha(\g,p)/2\le (3/4+o_k(1))\log n$.
Hence $F(\ell)\le e^{-(1/4-o(1))\log n}$ and so $\sum_{\ell\colon R\ge T}F(\ell)=o(1)$.

%In a similar manner, we bound the quantity $\PP(E_A \cap E_B)\PP(E_A)^{-2}$ in the case that $l \in [2,k^{2/3}]$. Here we are not able to use the %``flatness'' condition as we did before; however there is no danger, as there are not too many ``extra'' edges that could be in the intersection of %$A$ and $B$. We have that
%\begin{align*}
%\PP(E_A \cap E_B)/\PP(E'_A)^2 &= \PP(E'_A)^{-2}\sum_{0 \le L \le T } \PP(E_A \cap E_B \mid e(A\cap B) = L )\PP(e(A \cap B) = L) \\
%&\le  \PP(E'_A)^{-2}\sum_{0 \le L \le T } \PP(E'_A \cap E'_B \mid e(A\cap B) = L )\PP(e(A \cap B) = L) \\
%&= \sum_{0 \le L \le T } R_l(L)  \\
%&= \sum_{0 \le L < \g T } R_l(L) + \sum_{\g T \le L \le  T } R_l(L) \\
%&\le l^2\lambda^{T}e^{T\alpha(\g,p) + O_k(1) } .\end{align*}
%Here we have used the inequalities (\ref{equ:BoundonRl}) and (\ref{equ:BoundonRlneg}) again.
% So for $l \in [2,k^{2/3}]$ we have
%\begin{align}
% F(l) &\le l^2 t(l) \lambda^{T}2^{T\alpha(\g,p) + O_k(1)} \\
% &\le \left(2^{(l-1)\alpha(\g,p)/2  +  (l-1)/2\log \lambda  - \log n + 2\log k + o_k(1)} \right)^l \\
% &\le e^{-\frac{\ell}{2}\log n }, \label{equ:lsmallbound}
% \end{align}
%where we have used the fact that $k=O(\log n)$ and $\ell=o(k)=o(\log n)$.

%So finally, we may bound
%\[ F = \sum_{l =2}^{k-1} F(l) + F(k) \le \sum_{l = 2}^k n^{-\eps/2 \cdot l} + o_n(1) = o_n(1) ,\]
%where the inequality follows from the inequalities at (\ref{equ:llargebound}) and (\ref{equ:lsmallbound}).
%This finishes the proof of Lemma \ref{lem:VarianceCalc}.
\end{proof}

After these preparations, it is only a small step to finish the proof of Theorem~\ref{thm:MainTheorem}.

\begin{proof}[Proof of Theorem~\ref{thm:MainTheorem}.]
Let $\eps>0$ be given. The upper bound on $\omega_\g(G_{n,p})$ follows from Lemma~\ref{lem:FirstMoment}.
For the lower bound, assume $k\le\frac{2}{\alpha(\g,p)}(\log n-\log\log n+\log\frac{e\alpha(\g,p)}{2})+1-\eps$.
From Lemma~\ref{lem:firstmomentZ} we know that $\EE Z_k\to\infty$, so for sufficiently large $n$ we
have $\EE Z_k>0$ and thus we may apply Chebyshev's inequality to show that the quantity $\PP(X_k=0)$ is small. We have
\[
 \PP(X_k=0)\le\PP(Z_k=0)\le\PP(|Z_k-\EE Z_k|\ge\EE Z_k)\le\Var(Z_k)/\EE(Z_k)^2=F=o(1),
\]
where we have used the fact that every $\g$-flat set is a $\g$-quasi-clique for the first inequality.
The third inequality is Chebyshev's inequality and the bound on $F$ is the content
of Lemma~\ref{lem:VarianceCalc}.
\end{proof}

\section{Computational Experiments} \label{chap:CompExp}

Here, we note the bounds obtained from Theorem~\ref{thm:MainTheorem} are actually quite accurate in practice, even for relatively small values of $n$.
To illustrate, we performed a small set of computational experiments for graphs of size $n=50$ and $n=100$ and different values of $p$. For each pair $n,p$ we generated $100$ instances of graphs sampled according to the corresponding $G_{n,p}$ model. We have also selected various values of $\g$ ranging from $0.3$ to $0.9$. 

For each $\g$, $n$, $p$ in Table \ref{t_comp_exp} we report the minimum $\omega_{min}^\g$, maximum $\omega_{max}^\g$ and average $\omega_{avg}^\g$ cardinalities of the largest $\g$-quasi-cliques and compare this to $\omega_{th}^{\g}$, the ``theoretical'' value obtained from the formula in Theorem~\ref{thm:MainTheorem}. That is, 

\begin{equation} \label{equ:ThValue} \omega^{\g}_{th}(n) = \frac{2}{\alpha(\g,p)}\Big(\log n-\log\log n+\log\frac{e\alpha(\g,p)}{2}\Big) + \frac{1}{2} .\end{equation}

Observe that the obtained formula provides an accurate estimate of $\gamma$-quasi-clique
number $\omega_\gamma(G)$ in graph instances generated according to the binomial random graph $G_{n,p}$, even for relatively small values of $n$.

To identify the largest $\g$-quasi-clique in these experiments, we used the so-called feasibility check version of
formulation \textbf{F4} in \cite{veremyev2016exact} (or \textbf{AlgF4}). Previous experimental work has suggested this algorithm to be the best performing on instances generated from $G_{n,p}$. 

\begin{table}
\footnotesize
\setlength{\tabcolsep}{2mm}
\begin{tabular}{ccccc|ccccc|ccccc}
\hline
$\gamma$ & $\omega_{min}^\gamma$ & $\omega_{max}^\gamma$ &$\omega_{avg}^\gamma$ & $\omega_{th}^\gamma$ &
$\gamma$ & $\omega_{min}^\gamma$ & $\omega_{max}^\gamma$ &$\omega_{avg}^\gamma$ & $\omega_{th}^\gamma$ &
$\gamma$ & $\omega_{min}^\gamma$ & $\omega_{max}^\gamma$ &$\omega_{avg}^\gamma$ & $\omega_{th}^\gamma$ \\
\hline
\multicolumn{15}{c}{$n=50$} \\
\hline
\multicolumn{5}{c}{$p=0.20$} 	&	\multicolumn{5}{c}{$p=0.15$} &	\multicolumn{5}{c}{$p=0.10$} 	\\
\hline
0.9	&	4	&	5	&	4.95	&	5.72	&	0.9	&	3	&	5	&	4.12	&	5.06	&	0.9	&	3	&	5	&	3.27	&	4.39	\\
0.8	&	5	&	7	&	6.01	&	6.92	&	0.8	&	4	&	6	&	5.19	&	6.03	&	0.8	&	3	&	5	&	4.28	&	5.14	\\
0.7	&	6	&	8	&	7.2	&	8.44	&	0.7	&	5	&	8	&	6.02	&	7.26	&	0.7	&	3	&	6	&	5.05	&	6.09	\\
0.6	&	8	&	11	&	9.48	&	10.41	&	0.6	&	6	&	10	&	7.62	&	8.87	&	0.6	&	5	&	8	&	6.15	&	7.34	\\
0.5	&	10	&	15	&	12.58	&	12.64	&	0.5	&	8	&	12	&	9.85	&	10.99	&	0.5	&	6	&	10	&	7.8	&	9.05	\\
\hline
\hline
\multicolumn{15}{c}{$n=100$} \\
\hline
\multicolumn{5}{c}{$p=0.15$} 	&	\multicolumn{5}{c}{$p=0.10$} &	\multicolumn{5}{c}{$p=0.05$} 	\\
\hline
0.9	&	4	&	5	&	4.98	&	5.82	&	0.9	&	3	&	6	&	4.41	&	4.99	&	0.8	&	3	&	5	&	4.1	&	4.73	\\
0.85	&	4	&	6	&	5.6	&	6.4	&	0.8	&	5	&	7	&	5.23	&	5.92	&	0.6	&	5	&	7	&	5.72	&	6.65	\\
0.8	&	6	&	7	&	6.21	&	7.04	&	0.7	&	5	&	8	&	6.11	&	7.12	&	0.4	&	7	&	11	&	9.06	&	10.56	\\
0.75	&	6	&	8	&	6.95	&	7.78	&	0.6	&	7	&	10	&	7.74	&	8.75	&	0.3	&	11	&	16	&	12.77	&	14.44	\\
\hline
\end{tabular}
\caption{
Largest quasi-cliques in graphs generated according to $G_{n,p}$ model. For each $n,p$, the minimum $\omega_{min}^\gamma$, maximum $\omega_{max}^\gamma$ and average $\omega_{avg}^\gamma$ cardinalities of the largest quasi-cliques identified in 100 instances are reported. These values are compared against the values given by the formula for $\omega^{\g}_{th}$ at (\ref{equ:ThValue}).   
}
\label{t_comp_exp}
\end{table}

\bibliography{DenseSubgraphsBib}

\end{document}